\documentclass[12pt]{article}
\usepackage{amsmath,amsthm,amssymb,latexsym,color}
\usepackage{hyperref}

\theoremstyle{plain}
\newtheorem{theor}{Theorem}[section]

\newtheorem{claim}[theor]{Claim}
\newtheorem{prop}[theor]{Proposition}

\newtheorem{cor}[theor]{Corollary}

\theoremstyle{remark}

\def\la{\left\langle}
\def\r{\right}
\def\ra{\r\rangle}

\def\PP{{\mathbb P}}

\newcommand{\gam}{\gamma}

\newcommand{\eps}{\varepsilon}

\def\no{\|\cdot\|}

\def\R{{\mathbb R}}
\def\RR{{\mathbb{R}^n}}

\def\Event{{\mathcal E}}

\def\Prob{{\mathbb P}}
\def\Exp{{\mathbb E}}
\def\dist{{\rm dist}}
\def\spn{{\rm span}}
\def\col{{\rm \bf C}}
\def\row{{\rm \bf R}}

\title{Small ball probability for the condition number of random matrices}
\author{Alexander E. Litvak\footnote{University of Alberta} \and Konstantin Tikhomirov\footnote{Georgia Institute of Technology}
\and Nicole Tomczak-Jaegermann$^*$}

\newcommand\address{\noindent\leavevmode

\medskip

\noindent
Alexander E. Litvak
and Nicole Tomczak-Jaegermann,\\
Dept.~of Math.~and Stat.~Sciences,\\
University of Alberta, \\
Edmonton, AB, Canada, T6G 2G1.\\
\texttt{\small
e-mails:  aelitvak@gmail.com \, \, and \, \,
nicole.tomczak@ualberta.ca}

\medskip

\noindent
Konstantin Tikhomirov,\\
School of Mathematics,\\
Georgia Institute of Technology,\\
686 Cherry street,\\
Atlanta, GA 30332, USA.\\
\texttt{\small%
e-mail:  ktikhomirov6@gatech.edu}\\

}

\begin{document}

\maketitle

\begin{abstract}
Let $A$ be an $n\times n$ random matrix
with i.i.d.\ entries of zero mean, unit variance
and a bounded subgaussian moment.
We show that the condition number $s_{\max}(A)/s_{\min}(A)$
satisfies the small ball probability estimate
$$\Prob\big\{s_{\max}(A)/s_{\min}(A)\leq n/t\big\}\leq 2\exp(-c t^2),\quad t\geq 1,$$
where $c>0$ may only depend on the subgaussian moment.
Although the estimate can be obtained as a combination of known results
and techniques, it was not noticed in the literature before.
As a key step of the proof, we apply estimates for the singular values of $A$,
$\Prob\big\{s_{n-k+1}(A)\leq ck/\sqrt{n}\big\}\leq 2 \exp(-c k^2), \quad 1\leq k\leq n,$
obtained (under some additional assumptions) by Nguyen.
\end{abstract}

\noindent
{\small \bf AMS 2010 Classification:}
{\small
60B20,  15B52, 46B06, 15A18.
}

\noindent
{\small \bf Keywords: }
{\small
Random matrices, condition number, small ball probability, invertibility, smallest singular value.
}

\section{Introduction}

We say that a random variable $\xi$ has subgaussian moment bounded above by $K> 0$
if
$$
   \PP \{|\xi|\geq t\} \leq \exp\big(1-t^2/(2K^2)\big), \quad t\geq 0.
$$

Let $A$ be an $n\times n$ random matrix with i.i.d.\ entries of zero mean, unit variance
and subgaussian moment bounded above by $K$, and denote by $s_i(A)$, $1\leq i\leq n$,
its singular values arranged in non-increasing order. We will write $s_{\max}(A)$
and $s_{\min}(A)$ for $s_1(A)$ and $s_n(A)$, respectively.
Estimating the magnitude of the condition number,
$$\kappa(A)=s_{\max}(A)/s_{\min}(A),$$
is a well studied problem, with connections to numerical analysis and computation
of the limiting distribution of the matrix spectrum; we refer, in particular, to
\cite{TV} for discussion.
Since the largest singular value $s_{\max}(A)$ is strongly concentrated
(see the proof of Corollary~\ref{cond} below),
estimating $\kappa(A)$ is essentially reduced to estimating $s_{\min}(A)$ from above and below.

The main result of \cite{RV square} provides small ball probability estimates for $s_{\min}(A)$ of the form
$$\Prob\big\{s_{\min}(A)\leq t/\sqrt{n}\big\}\leq Ct+e^{-c n},\quad t\leq 1,$$
for some $C,c>0$ depending only on the subgaussian moment.
It seems natural to investigate the complementary regime --- the large deviation estimates for $s_{\min}(A)$.
It was shown in \cite{RV sm ball} that
$$
\Prob\big\{s_{\min}(A)\geq t/\sqrt{n}\big\}\leq \frac{C\ln t}{t}+e^{-c n},\quad t\geq 2
$$
(see also \cite{Katya} for an extension of this result to distributions with no assumptions on moments higher than $2$).
The probability estimate was improved in \cite{NV} to
$$
\Prob\big\{s_{\min}(A)\geq t/\sqrt{n}\big\}\leq e^{-ct},\quad t\geq 2,
$$
for $c>0$ depending only on the subgaussian moment (see also \cite{Wei} for a generalization to intermediate singular values).
The existing results on the distribution of the singular values of random Gaussian matrices
\cite{Ed,Szarek} suggest that the optimal dependence on $t$ in the exponent on the right hand side is quadratic,
i.e.\ the variable $\sqrt{n}\,s_{\min}(A)$ is subgaussian.
Specifically, it is shown in \cite{Szarek} that $s_{\min}(G)$ for the standard $n\times n$
Gaussian matrix $G$ satisfies two-sided estimates
$$
\exp(-C t^2)\leq \Prob\big\{s_{\min}(G)\geq t/\sqrt{n}\big\}\leq \exp(-c t^2),\quad t\geq C_1,
$$
where $C, C_1, c>0$ are some universal constants.
The main result of our note provides matching upper estimate for matrices with subgaussian entries:

\begin{theor}\label{th: smin}
Let $A$ be an $n\times n$ random matrix with i.i.d.\ entries of zero mean, unit variance, and
subgaussian moment bounded above by $K>0$.
Then the smallest singular value $s_{\min}(A)$ satisfies
$$
\Prob\big\{s_{\min}(A)\geq t/\sqrt{n}\big\}\leq 2\exp(-c t^2),\quad t\geq 1,
$$
where $c>0$ is a constant depending only on $K$.
\end{theor}


As a simple corollary of the theorem, we obtain small ball probability estimates for the condition number:
\begin{cor}
\label{cond}
Let $A$ be an $n\times n$ random matrix with i.i.d.\ entries of zero mean, unit variance, and
subgaussian moment bounded above by $K>0$.
Then the condition number $\kappa(A)$ satisfies
$$
\Prob\big\{\kappa(A)\leq n/t\big\}\leq 2\exp(-c t^2),\quad t\geq 1,
$$
where $c>0$ is a constant depending only on $K$.
\end{cor}

Theorem~\ref{th: smin} is a consequence of the following theorem, which is of independent interest.

\begin{theor}\label{h-s-norm}
Under conditions of Theorem~\ref{th: smin} one has
$$
\Prob\big\{\|A^{-1}\|_{HS} \leq \min(n/t, \, \sqrt{n/t})\big\}\leq 2\exp(-c t^2),\quad t\geq 0,
$$
where $c>0$ is a constant depending only on $K$.
\end{theor}

The proof of Theorem~\ref{h-s-norm} uses, as a main step,
the estimates
$$\Prob\big\{s_{n-k+1}(A)\leq ck/\sqrt{n}\big\}\leq 2 \exp(-c k^2),\quad \quad 1\leq k\leq n,$$
for the singular values of the matrix $A$.
These estimates, based on the restricted
invertibility of matrices and certain averaging arguments,
were recently obtained by Nguyen \cite{Nguyen} under some additional assumptions (which will be discussed in the next section).

\section{Preliminaries}

Given a matrix $A$, it singular values $s_i=s_i(A)$, $i\geq 1$, are square
roots of eigenvalues of $AA^*$. We always assume that $s_1\geq s_2\geq \ldots$
By $\|A\|$ and $\|A\|_{HS}$ we denote the operator $\ell_2\to \ell_2$ norm of $A$ (also
called the spectral norm) and the Hilbert--Schmidt norm respectively. Note that
$$
   \|A\| = s_1 \quad \quad \mbox{ and } \quad \quad \|A\|_{HS}^2=\sum_{i\geq 1} s_i^2.
$$
The columns and rows of $A$ are denoted by $\col_i(A)$ and $\row_i(A)$, $i\geq 1$, respectively.
Given $J\subset [m]$, the coordinate projection in $\R^m$ onto $\R^J$ is denoted by $P_J$.
For convenience, we often write $A_J$ instead of $AP_J$. Given $m\geq 1$, the identity
operator $\R^\ell \to \R^\ell$ we denote by $I_m$.
Given $x, y\in \RR$ by $\la x, \cdot\ra y$  we denote the operator $z\mapsto \la x, z\ra y$
(in the literature it is often denoted by $x\otimes y$ or $y x^\top$). The canonical
Euclidean norm in $\R^m$ is denoted by $\no_2$ and the unit Euclidean sphere by $S^{m-1}$.

As the most important part of our argument, we will use the following result.
\begin{theor}\label{th: sing values}
Let $A$ be an $n\times n$ random matrix with i.i.d. entries of zero mean, unit variance, and
subgaussian moment bounded above by $K>0$.
Then for any $1\leq k\leq n$ one has
$$\Prob\big\{s_{n-k+1}(A)\leq ck/\sqrt{n}\big\}\leq 2 \exp(-c k^2),$$
where $c>0$ is a constant depending only on $K$.
\end{theor}
The above theorem, up to some minor modifications, was proved by Nguyen in \cite{Nguyen}.
Specifically, in the case $k\geq C\log n$, the theorem follows from \cite[Theorem~1.7]{Nguyen}
(or \cite[Corollary~1.8]{Nguyen})
if one additionally assumes either that the entries of $A$ are uniformly bounded by a constant,
or that the distribution density of the entries is bounded. Removing these conditions
requires a minor change of the proof in \cite{Nguyen}.
Further, in the case $k\leq C\log n$, the above result (in fact, in a stronger form) is stated as formula (4)
in \cite[Theorem~1.4]{Nguyen}. However, \cite[Theorem~3.6]{Nguyen}, which is used to derive \cite[formula~(4)]{Nguyen},
provides a non-trivial probability estimate only for the event $\{s_{n-k+1}(A)\leq c_\gamma k^{1-\gamma}/\sqrt{n}\}$
(for any given $\gamma\in(0,1)$ and $c_\gamma$ depending on $\gamma$), see \cite[formula~(31)]{Nguyen}.
Again, a minor update of the argument of \cite{Nguyen} provides the result needed for our purposes.
In view of the above and for the reader's convenience, we  provide a proof of Theorem~\ref{th: sing values}
in the last section.

The following result was proved in \cite{SS} as an extension of the classical Bourgain--Tzafriri
restricted invertibility theorem \cite{BT}. With worse dependence on $\varepsilon$,
the theorem was earlier proved in \cite{V restricted}. See also a recent paper \cite{NY} for further
improvements and discussions.
\begin{theor}[{\cite{SS}}]\label{restr}
Let $T$ be $n\times n$ matrix. Then for any $\varepsilon\in(0,1)$
there is a set $J\subset[n]$ such that
$$
\ell:=|J|\geq \bigg\lfloor\frac{\varepsilon^2 \|T\|_{HS}^2}{\|T\|^2}\bigg\rfloor \quad\quad  \mbox{ and } \quad \quad
  s_\ell (T_J)\geq \frac{(1-\varepsilon)\|T\|_{HS}}{\sqrt{n}}.
$$
\end{theor}

We  will use two following results by Rudelson--Verhsynin. The first one was
one of the key ingredients in estimating the smallest singular value of
rectangular matrices. The second one is an immediate consequence of
the Hanson--Wright inequality \cite{HW, W} generalized in \cite{RV HW}.

\begin{theor}[{\cite{RV}, Theorem 4.1}]\label{RV-dist}
Let $X$ be a vector in $\R^n$, whose coordinates are i.i.d.\ mean-zero, subgaussian
random variables with unit variance. Let $F$ be a random subspace in $\R^n$ spanned
by $n-\ell$ vectors, $1\leq \ell\leq  c' n$, whose coordinates are i.i.d.\ mean-zero, sub-Gaussian random
variables with unit variance, jointly independent with $X$. Then, for every $\eps >0$, one has
$$
\PP\big\{\dist  (X, F)\leq \eps \sqrt{\ell} \big\}\leq (C\eps)^\ell + \exp(-cn).
$$
where $C>0$,  $c, c'\in (0, 1)$ are constants depending only on the subgaussian moments.
\end{theor}

\begin{theor}[{\cite[Corollary 3.1]{RV HW}}]\label{RV-HW}
Let $X$ be a vector in $\R^n$, whose coordinates are i.i.d.\ mean-zero random variables
with unit variance and with  subgaussian moment bounded by $K$.
 Let $F$ be a fixed subspace in $\R^n$ of dimension $n-\ell$.
Then, for every $t >0$, one has
$$
 \PP\big\{| \dist  (X, F)-\sqrt{\ell}|\geq t \big\}\leq 2 \exp(-c t^2/K^4).
$$
where $c>0$ is an absolute constant.
\end{theor}

We will also need the following standard claim, which can be proved by integrating
the indicator functions (see e.g., {\cite[Claim~3.4]{Nguyen}}, cf. {\cite[Claim~4.9]{LLTTY}}).

\begin{claim}\label{even}
Let $\alpha, p\in (0, 1)$. Let $\Event$ be an event.
Let $Z$ be a finite index set, and
$
\{\Event_z \} _{z\in Z}
$
be a collection of $|Z|$ events satisfying $\PP(\Event_z)\leq p$ for every $z\in Z$.
Assume that at least $\alpha |Z|$ of events $\Event_z$ hold whenever the  event  $\Event$
occurs. Then $\PP(\Event)\leq p/\alpha$.
\end{claim}

\section{Proofs of main results}
{The condition number}

\begin{proof}[Proof of Theorem~\ref{th: smin}]
In the case  $t>n$  we have
\begin{align*}
\Prob\big\{ s_{\min}(A)\geq t/\sqrt{n}
\big\}
&= \Prob\Big\{s_1(A^{-1})\leq \sqrt{n}/t \Big\}\leq
\Prob\Big\{\sum_{i=1}^n s_i(A^{-1})^2\leq n^2/t^2\Big\}
\end{align*}
and the result follows from Theorem~\ref{h-s-norm}.

Now we consider the case $1\leq t\leq n$.
Let $L\geq 1$ be a parameter which we will choose later.
Then
\begin{align*}
\Prob\big\{s_{\min}(A)\geq t/\sqrt{n}\big\}
&= \Prob\big\{s_{1}(A^{-1})\leq \sqrt{n}/t\big\}
\\&\leq \Prob\Big\{s_1(A^{-1})^2\leq n/t^2\;\;\mbox{ and }\;\;\sum_{i\geq \lceil t\rceil} s_i(A^{-1})^2\geq L n/t\Big\}\\
&\hspace{2cm}+\Prob\Big\{s_1(A^{-1})^2\leq n/t^2\;\;\mbox{ and }\;\;\sum_{i\geq  \lceil t\rceil} s_i(A^{-1})^2< L n/t\Big\}\\
&\leq \Prob\Big\{\sum_{i\geq \lceil t\rceil} s_i(A^{-1})^2\geq L n/t\Big\}
+\Prob\Big\{\sum_{i=1}^n s_i(A^{-1})^2\leq n/t+ L n/t\Big\}.
\end{align*}
For the first summand in the last expression, we apply Theorem~\ref{th: sing values}.
Since $\sum_{i=\lfloor t\rfloor}^\infty \frac{1}{i^2}\leq \frac{2}{t}$,
we obtain
\begin{align*}
\Prob\Big\{\sum_{i=\lceil t\rceil}^n s_i(A^{-1})^2\geq L n/t\Big\}
&\leq \sum\limits_{i=\lceil t\rceil}^n \Prob\big\{s_i(A^{-1})^2\geq L n/(2 i^2)\big\}\\
&= \sum\limits_{i=\lceil t\rceil}^n \Prob\big\{s_{n-i+1}(A)\leq \sqrt{2} i/\sqrt{L n}\big\}.
\end{align*}
Choosing $L$ so that $\sqrt{2/L}$ is equal to the constant from Theorem~\ref{th: sing values},
we get
$$
\sum\limits_{i=\lfloor t\rfloor}^n \Prob\big\{s_{n-i+1}(A)\leq \sqrt{2} i/\sqrt{L n}\big\}
\leq 2  \sum\limits_{i=\lfloor t\rfloor}^n \exp(-c i^2)\leq 3 \exp(-c' t^2)
$$
for some $c'>0$ depending only on $K$.
The bound on the second summand follows from Theorem~\ref{h-s-norm} applied with $t/(L+1)$ instead of
$t$. This completes the proof.

\end{proof}

\medskip

\begin{proof}[Proof of Corollary~\ref{cond}]
Theorem~\ref{RV-HW} implies that there exists an absolute constant
$c_1>0$ depending only on $K$  such that for every $i\leq n$
\begin{equation*}
  \PP(\|\col_i(A) \|_2\leq  \sqrt{n}/2 )\leq \exp(- c_1  n)
\end{equation*}
(this can be shown by direct calculations as well, see e.g. Fact~2.5 in \cite{LPRT}).
Since the entries of $A$ are independent, we obtain
\begin{equation*}
  \PP(\|A \|\leq  \sqrt{n}/2 )\leq \prod_{i=1}^n
  \PP(\|\col_i(A) \|_2\leq  \sqrt{n}/2 )\leq \exp(- c_1  n^2).
\end{equation*}
Note that if $\|A \|\geq  \sqrt{n}/2$ and $\kappa(A)\leq  n/2t$ then
$s_n(A)= \|A\|/\kappa(A)\geq t/\sqrt{n}$.
Therefore, by Theorem~\ref{th: smin},
$$
  \PP\{\kappa(A)\leq  n/2t\} \leq 2\exp(-c t^2) + \exp(- c_1  n^2).
$$
By adjusting constants, this implies the conclusion  for $t\leq n$.
Since $\kappa(A)\geq 1$, the case $t>n$ is trivial.
\end{proof}

\medskip

\begin{proof}[Proof of Theorem~\ref{h-s-norm}]
Adjusting the constant in the exponent if
needed, without loss of generality, we assume  that $t\geq C_0$, where $C_0>0$ is a
large enough constant depending only on $K$.
Denote
$$
 \Event_0:=\bigg\{\sum_{i=1}^n s_i(A^{-1})^2\leq n/t\bigg\}.
$$

We first consider the case $t\leq n$.
Applying the negative second moment identity (see e.g. Exercise 2.7.3 in \cite{Tao}),
$$
\sum_{i=1}^n s_i(A^{-1})^2=\sum_{i=1}^n \dist\big(\col_i(A),\spn\{\col_j(A),\; j\neq i\}\big)^{-2},
$$
we observe that on the event $\Event_0$,
$$\big|\big\{i\leq n:\;\dist\big(\col_i(A),\spn\{\col_j(A),\; j\neq i\}\big)\geq \sqrt{t/2}\big\}\big|\geq n/2.$$
For each subset $I\subset [n]$ of cardinality $k\leq n/2$
(the actual value of $k$ will be defined later), let ${\bf 1}_{I}$ be the indicator of the event
$$
 \big\{\dist\big(\col_i(A),\spn\{\col_j(A),\; j\in [n]\setminus I\}\big)\geq \sqrt{t/2}
\mbox{ for all }i\in I\big\}.
$$
Then, in view of the above, everywhere on the event $\Event_0$
we have
$$
 \sum_{I\subset[n],\;|I|=k} {\bf 1}_{I}\geq {\lceil n/2\rceil \choose k}\geq \bigg(\frac{n}{2k}\bigg)^k
  \geq  (2e)^{-k} {n\choose k}.
$$
Hence, by Markov's inequality and permutation invariance of the matrix distribution,
$$
\Prob(\Event_0)
\leq  (2e)^{k} \, \Exp\, {\bf 1}_{[k]}.
$$
As the last step of the proof, we estimate the expectation of ${\bf 1}_{[k]}$ (with a suitable choice of $k$).
In view of independence and equidistribution of the matrix columns, we have
$$
\Exp\, {\bf 1}_{[k]}=\left(
\Prob\big\{
\dist\big(\col_1(A),\spn\{\col_j(A),\; j\in [n]\setminus [k]\}\big)\geq \sqrt{t/2}
\big\}\r)^k.
$$
Choose $k:=\lfloor t/4\rfloor\leq n/2$ and denote
$$D:=\dist\big(\col_1(A),\spn\{\col_j(A),\; j\in [n]\setminus [k]\}\big).$$
Using independence of columns of the matrix  $A$ and
applying Theorem~\ref{RV-HW} with $\ell=k$ and
$F=\spn\{\col_j(A),\; j\in [n]\setminus [k]\}$,
 we obtain
$$
 \Prob \Big\{ D \geq \sqrt{t/2} \Big\} \leq
 \Prob\Big\{ D - \sqrt{k} \geq (\sqrt{2}-1)\, \sqrt{t/4} \Big\} \leq 2\exp(-\bar c\, t)
$$
for some $\bar c>0$ depending only on $K$.
Hence,
$$
\Prob(\Event_0) \leq
(2e)^{k} \,2^k\,\exp(-\bar c\, t\,k)\leq \exp(-\bar c t^2/16),
$$
provided that $t$ is larger than a certain constant depending only on $K$.
This implies the desired result for $t\leq n$.

In the case $t>n$ we essentially repeat the argument along the same lines. Define
$$
\Event_0':=\bigg\{\sum_{i=1}^n s_i(A^{-1})^2\leq n^2/t^2\bigg\}.
$$
Observe that on the event $\Event_0'$,
$$
  \big|\big\{i\leq n:\;\dist\big(\col_i(A),\spn\{\col_j(A),\; j\neq i\}
  \big)\geq t/\sqrt{2n}\big\}\big|\geq n/2.
$$
Repeating the above computations with the same notation and with $k=\lfloor n/4\rfloor$ we obtain
 $$
 \Prob \Big\{ D \geq t/\sqrt{2n} \Big\} \leq
 \Prob\Big\{ D - \sqrt{k} \geq t/(5 \sqrt{n}) \Big\} \leq 2\exp(-\bar c\, t^2/n),
$$
which leads to
$$
\Prob(\Event_0') \leq
(2e)^{k} \,2^k\,\exp(-\bar c\, k t^2/n )\leq \exp(-\bar c t^2/16),
$$
provided that $t> Cn$ for large enough $C$ depending only on $K$. For
$n< t\leq Cn$ the result follows by adjusting the constants.
\end{proof}

\section{Small ball estimates for singular values}

The goal of this section is to prove Theorem~\ref{th: sing values}.
As we have noted, the argument essentially reproduces that of \cite{Nguyen}.
An important part of the proof is the use of restricted invertibility
(see also \cite{Cook} and \cite{RT} for some recent applications of restricted invertibility
in the context of random matrices).

\medskip

We will use a construction from \cite{Nguyen}.
Given an integer $k$ and an $n\times n$  matrix $A$ define a $k\times n$ matrix $Z=Z(A, k)$ in the following way.
Consider singular value decomposition $A=\sum_{i=1}^n s_i \la v_i, \cdot\ra w_i$, where $s_i=s_i(A)$ are
singular values of $A$ (arranged in non-increasing order) and $\{v_i\}_i$, $\{w_i\}_i$ are two
orthonormal systems in $\RR$. For $i\leq k$ denote $z_i= v_{n-i+1}$. Let $Z$ be the matrix whose
rows are $\row_i(Z)= z_i$. Clearly, the rows of $Z$ are orthonormal and for every $i\leq k$,
\begin{equation}\label{prop1}
  \|Az_i\|_2 = s_{n-i+1}\leq s_{n-k+1}.
\end{equation}
Moreover,
\begin{equation*}
 \|Z\|=1\quad \quad \mbox{ and } \quad \quad \|Z\|_{HS}=\sqrt{k}.
\end{equation*}
The matrix $Z$ is not uniquely defined when some of the $k$ smallest singular values of $A$ have non-trivial multiplicity;
we will however assume that for each realization of $A$, a single admissible $Z$ is chosen in such a way that $Z$
is a (measurable) random matrix.

\medskip

\subsection{Proof of Theorem~\ref{th: sing values}, the case $k\geq \ln n$}
\label{pf-1}

Let $C, c, c'$ be constants from Theorem~\ref{RV-dist}.
Let $\gamma =\sqrt{c'}$. Note that $C, c, c', \gamma$ depend only on $K$.
Let $Z=Z(A, k)$ be the $k\times n$ matrix constructed above.
Applying Theorem~\ref{restr} to $Z$ (one can add zero rows to make it an $n\times n$ matrix),
there exists $J\subset [n]$ such that
$$
 |J| = \ell:=\lfloor \gam^2 k \big\rfloor\leq c' k \quad \quad \mbox{ and } \quad \quad
 s_\ell (Z_J)\geq (1-\gam ) \sqrt{k/n}.
$$
Fix a (small enough, depending on $K$) constant $c_0>0$. Define the event
$$
 \Event_k := \big\{s_{n-k+1}(A)\leq c_0 k/\sqrt{n}\big\}.
$$
Consider the $n\times k$  matrix $B=A Z^\top$. Using property \eqref{prop1}, on the event  $\Event_k$,
we have for every $i\leq k$,
$$
 \|\col _i (B) \|_2=  \|A z_i\|_2\leq c_0 k/\sqrt{n},
$$
hence $\|B\|_{HS}\leq c_0 k^{3/2}/\sqrt{n}$.
Now, since $s_\ell (Z_J)>0$, there exists a $k\times \ell$ matrix $M$ such that $Z^\top_J M=I_\ell$. Then
$$
   \|M\| =1/s_\ell (Z)\leq (1-\gam )^{-1} \sqrt{n/k}.
$$
Therefore,
$$
 \|B M\|_{HS}\leq \|B\|_{HS}\, \|M\|\leq  c_0 (1-\gam )^{-1} k.
$$
Writing $B= A_J (Z_J)^\top + A_{J^c} (Z_{J^c})^\top$, we also have
$BM=  A_J  + A_{J^c} (Z_{J^c})^\top M$.
Next denote
$$
 F=F(A, J):=\spn\{ \col_i(A_{J^c})\}_{i\in J^c},
$$
and let $P$ be the orthogonal projection on $F^\perp$.
Then, on the event $\Event _k$,
$$
 c_0^2 (1-\gam )^{-2} k^2 \geq \|P B M\|_{HS}^2 \geq \|P  A_J\|_{HS}^2 =
 \sum _{i\in J} \|P \, \col_i(A_J)\|^2_2 =  \sum _{i\in J} \dist ^2 (\col_i(A), F) .
$$
Therefore, for at least $\ell/2$ indices $i\in J$, one has
$$
   \dist  (\col_i(A), F)  \leq \sqrt{2} c_0 (1-\gam )^{-1} k/\sqrt{\ell}
   \leq 2 c_0 \sqrt{\ell}/((1-\gam )\gam^2) .
$$
Note that the subspace $F$ is spanned by $n-\ell$ random vectors,
it is independent of columns $\col_i(A)$, $i\in J$, and
that columns of $A$ are independent. Therefore, by Theorem~\ref{RV-dist} and the union bound
we obtain
\begin{align*}
  \PP(\Event _k)&\leq \sum _{J\subset[n]\atop |J|=\ell}  \sum _{J_1\subset J \atop |J|=\lceil \ell/2\rceil}
  \PP\Big\{ \forall i\in J_1 \,\, \,\, \dist  (\col_i(A), F)  \leq 2 c_0 \sqrt{\ell}/((1-\gam )\gam^2)\Big\}
  \\& \leq
  {n\choose \ell}\, 2^\ell \, \left((2 C c_0 /((1-\gam )\gam^2))^{\ell} + \exp(-cn) \r)^{\ell/2}
  \\& \leq
  \left(\frac{4en}{\ell} \max\left\{\left(\frac{\sqrt{2 C c_0}}{\gam \sqrt{1-\gam}}\r)^\ell,\, \exp(-cn/2)\r\}\r)^{\ell}
\end{align*}
 Choosing
 small enough  $c_0$ and using $k\geq \ln n$, we obtain
 $
   \PP(\Event _k) \leq \exp(-c_3  \ell^2) ,
 $
 where $c_3>0$ depends only on $K$. By adjusting constants this proves the desired result for $k\geq \ln n$.
\qed

\subsection{Proof of Theorem~\ref{th: sing values}, the case $k\leq \ln n$}

Let $A$ be as in Theorem~\ref{th: sing values}.
It is well known (see e.g. Fact~2.4 in \cite{LPRT}) that there is an absolute constant
$C_1>0$ such that
\begin{equation}\label{bdd}
  \PP\big\{\|A\|\leq C_1 K\sqrt{n} \big\}\geq 1- e^{-n}.
\end{equation}
Let $\Event_{bd}$ denote the event from this equation.
Further, from \cite[Theorem~1.5]{RV no-gaps} one infers that
for any $\gamma>0$ there are $\gamma_1,\gamma_2,\gamma_3>0$ depending only on $\gamma$ and $K$ such that,
denoting
\begin{align*}
\Event_{inc}(\gamma) :=
\big\{&\forall x \in S^{n-1}  \, \mbox{ with }  \, \|Ax \|_2\leq \gamma_1 \sqrt{n}, \, \,
\forall I\subset [n]\\
&\mbox{ with } \,  |I|\geq \gamma n   \, \mbox{ one has }  \,
\|P_I x\|_2\geq \gamma_2 \big\},
\end{align*}
the event satisfies
\begin{equation}\label{inc}
\Prob(\Event_{inc}(\gamma))\geq 1-2 e^{-\gamma_3 n}.
\end{equation}


The following statement was proved by  Nguyen ({\cite[Corollary~3.8]{Nguyen}}).

\begin{prop}\label{Ng-38}
For any $K>0$ there are $C, c_1,c_2,\gamma>0$ depending only on $K$ with the following property.
Let $A$ be an $n\times n$ random matrix with i.i.d.\ entries of zero mean, unit variance, and
subgaussian moment bounded above by $K$. Let $2\leq k\leq n/(C\ln n)$, and let
the random $k\times n$ matrix $Z=Z(A, k)$ be defined as above.
Then everywhere on the event $\big\{s_{n-k+1}(A)\leq c_1 k/\sqrt{n}\big\}\cap \Event_{inc}(\gamma)\cap \Event_{bd}$
one has
$$
\big|\big\{J\subset[n]:\;|J|=\lfloor k/2\rfloor,\;s_{\lfloor k/2\rfloor}(Z_J)\geq c_1\sqrt{k/n}\big\}\big|
\geq c_2^{k\ln k}\, n^{\lfloor k/2\rfloor}.
$$
\end{prop}

Now assume that $k\leq \ln n$. Without loss of generality we may also assume that $k$
is bounded below by a large constant.
Let $C, c, c'$ be constants from Theorem~\ref{RV-dist} and $c_1, c_2,\gamma$ from Proposition~\ref{Ng-38}.
Fix for a moment any realization of $A$ from the event
$\big\{s_{n-k+1}(A)\leq c_0 k/\sqrt{n}\big\}\cap \Event_{inc}(\gamma)\cap \Event_{bd}$,
where $c_0\in(0,c_1]$ will be chosen later.
Let $\ell :=\lfloor k/2\rfloor$ and
$$
\mathcal{J} :=\big\{J\subset[n]:\;|J|=\lfloor k/2\rfloor,\;s_{\lfloor k/2\rfloor}(Z_J)\geq c_1\sqrt{k/n}\big\}.
$$
Fix $J\in  \mathcal{J}$ and repeat the procedure used in Subsection~\ref{pf-1} with $J$ and $\ell$.
We obtain that for at least $\ell/2$ indices $i\in J$, one has
\begin{equation}\label{one}
   \dist  (\col_i(A), F)  \leq \sqrt{2} c_0 k/(c_1 \sqrt{\ell})
   \leq 4 c_0\sqrt{\ell}/c_1,
\end{equation}
where $F=\spn\{ \col_i(A_{J^c})\}_{i\in J^c}$.
For any fixed subset $J\subset[n]$ of cardinality $\ell$ consider the event
$$
\Event_J:=\big\{\mbox{for at least $\ell/2$ indices $i\in J$ inequality \eqref{one} holds}\big\}.
$$
Applying Theorem~\ref{RV-dist} and the union bound we observe
\begin{align*}\PP(\Event _J)&\leq
   2^\ell \, \left((4 c_0 C /c_1)^{\ell} + \exp(-cn) \r)^{\ell/2}
  \leq
  \left(4\, \max\left\{\left(4 c_0 C/c_1 \r)^\ell,\, \exp(-cn)\r\}\r)^{\ell/2}.
\end{align*}
Choosing $c_0$ to be small enough we obtain that $\PP(\Event _J)\leq   \exp(-c_4  k^2)$,
where $c_4>0$ depends only on $K$. Combining this with Claim~\ref{even} and Proposition~\ref{Ng-38}
we obtain
$$
 \PP\big(\big\{s_{n-k+1}(A)\leq c_0 k/\sqrt{n}\big\}\cap \Event_{inc}(\gamma)\cap
\Event_{bd}\big)\leq c_2^{- k\ln k } \exp(-c_4  k^2) \leq \exp(-c_5  k^2)
$$
provided that $k\geq C_2$, where $C_2\geq 1 \geq c_5>0$ are constants depending on
on $K$ only. By \eqref{bdd} and \eqref{inc} this completes the proof in the case $k\leq \ln n$.
 \qed

\bigskip

\bigskip

\noindent
{\bf Acknowledgment.} The authors are grateful to the anonymous referee for careful reading
and valuable suggestions that have helped to improve the presentation.  The second named
author would like to thank the Department of Mathematical and Statistical Sciences,
University of Alberta, for ideal working conditions.

\address

\end{document}